\documentclass [12pt]{article}
\usepackage{amssymb,amsmath,amsthm,amscd}
\newtheorem{theorem}{Theorem}
\newtheorem{lemma}{Lemma}
\newtheorem{remark}{Remark}
\newtheorem{proposition}{Proposition}
\newtheorem{corollary}{Corollary}
\newtheoremstyle{neosn}{0.5\topsep}{0.5\topsep}{\rm}{}{\sc}{.}{ }{\thmname{#1}\thmnumber{ #2}\thmnote{ {\mdseries#3}}}
\theoremstyle{neosn}
\newtheorem{definition}{Definition}

\newcommand{\tr}{\,\mathrm{tr}\,}

\newcommand{\ad}{\,\mathrm{ad}\,}

\newcommand{\GL}{\,\mathrm{GL}\,}
\newcommand{\SL}{\,\mathrm{SL}\,}
\newcommand{\PSL}{\,\mathrm{PSL}\,}

\newcommand{\Sp}{\,\mathrm{Sp}\,}
\newcommand{\Spin}{\,\mathrm{Spin}\,}

\newcommand{\SO}{\,\mathrm{SO}\,}
\newcommand{\Aut}{\,\mathrm{Aut}\,}

\newcommand{\Hom}{\,\mathrm{Hom}\,}

\begin{document}
\begin{center}

{\Large {\bf Automorphisms of  Chevalley groups
over commutative rings }}

\bigskip
\bigskip

{\large \bf E.~I.~Bunina}

\bigskip

{\bf Bar Ilan University}

\end{center}
\bigskip

\begin{center}

{\bf Abstract.}

\end{center}

In this paper we prove that every automorphism of a
Chevalley group (or its elementary subgroup) with root system of rank $>1$ over a commutative ring  (with $1/2$ for the systems $\mathbf A_2$, $\mathbf F_4$, $\mathbf B_l$, $\mathbf C_l$; with $1/2$ and $1/3$ for the system~$\mathbf G_2$) is standard, i.\,e., it is a composition of ring, inner, central and graph automorphisms. This result finalizes description of automorphisms of Chevalley groups. However the restrictions on invertible elements can be a topic of further considerations. We provide also some model-theoretic applications of this description.

\bigskip

\section{Introduction}\leavevmode

\subsection{Automorphisms and isomorphisms of classical linear groups}\leavevmode

Automorphisms and isomorphisms of linear groups are studied by mathematicians from the beginning of XX century.
First papers on automorphisms and isomorphisms of linear groups appeared already in the beginning of the 20th century. In particular, in the  paper by Schreier and van der Warden~\cite{1928} they 
 described all automorphisms of the group $\PSL_n$ $(n\geqslant 3)$ over an arbitrary field. Later on,  Hua \cite{Hua_Agatha}  generalized this method and applied it to the description of automorphisms of symplectic groups over a field of characteristic $\ne 2$.
Diedonne \cite{D} (1951) and Rickart \cite{R1950} (1950) introduced the involution method, and described
automorphisms of the group $\GL_n$ ($n\geqslant 3)$ over a skew field, and then also of unitary and symplectic groups over skew fields of characteristic $\ne 2$~\cite{Rickart_Agatha}.

The first step towards the description of automorphisms of classical groups over rings was made by Hua and Reiner~\cite{HR}. They dealt with the case $\GL_n(\mathbb Z)$. This result 
 was extended to non-commutative principal ideal domains  by Landin and Reiner in~\cite{Landin-Reiner_Agatha} and  by Yan Shi-jian in~\cite{Shi-jian_Agatha}.

The methods of the papers mentioned above were based mostly on studying involutions in the corresponding linear groups.

O'Meara in 1976 invented very different (geometrical) method, which did not use involutions. By its aid, O"Meara described automorphisms of the group $\GL_n$ ($n\geqslant 3$) over domains~\cite{O'M2} and  automorphisms  of symplectic groups of a special form over fields (so-called \emph{groups rich in transvections})~\cite{O'M_Agatha}.
Independently, Yan Shi jian in~\cite{Shi-jian_Agatha} described  automorphisms of the group $E_n(R)$, $n\geqslant 3$, where $R$ is a domain  of characteristic $\neq 2$ using the involution method.

In the paper~\cite{McDonald_Agatha} Pomfret and MacDonald studied automorphisms of the groups $\GL_n$, $n \geqslant 3$, over a commutative local ring with~$1/ 2$. Further on, Waterhouse in~\cite{v47} obtained a description of automorphisms of the group $\GL_n$, $n \geqslant 3$, over arbitrary commutative rings with~$1/2$.

In 1982 Petechuk~\cite{v12} described automorphisms of the groups $\GL_n$, $\SL_n$ ($n\geqslant 4$)
over arbitrary commutative rings. If $n=3$, then automorphisms of  linear groups are not always standard~\cite{Petechuk2}.
They are standard either  if in a ring $2$ is invertible, or if a ring is a domain, or it is a semisimple ring.

McQueen and McDonald in~\cite{McQueen_Agatha} obtained the description of automorphisms of the groups $\Sp_n$, $n\geqslant 6$ over commutative local rings with $1/2$.  Continuing research in this direction, in 1980  Petechuk in~\cite{Petechuk1_Agatha} studied automorphisms of symplectic  groups over arbitrary commutative local rings. In 1982 he extended description of automorphisms to the case 
 $\Sp_n(R)$, $n \geqslant 6$, over arbitrary commutative ring~$R$, using the localization method, see~\cite{Petechuk2_Agatha}.

Isomorphisms of the groups $\GL_n(R)$ and $\GL_m(S)$ over arbitrary associative rings with $1/2$
for $n,m\geqslant 3$ were described in 1981 by Golubchik and \,Mikhalev~\cite{GolMikh1} and independently by
Zelmanov~\cite{v11}. In 1997 Golubchik described isomorphisms between these groups for $n,m\geqslant 4$,
over arbitrary associative rings with~$1$~\cite{Golub}.

In 1983 Golubchik and Mikhalev in~\cite{v8} studied isomorphisms of unitary linear groups over arbitrary associative rings with~$1/2$, with some conditions for the dimension of the group and the rank of  the form.
For the case when $n=2k$ and the hyperbolic rank of the form~$Q$ is maximal,  the automorphism of $U_n(R, Q)$, $k\geqslant 3$, were independently classified in 1985 by Zelmanov, see~\cite{v11}.

\subsection{Automorphisms and isomorphisms of Chevalley groups}\leavevmode

In  50-th  years of the previous century Chevalley, Steinberg and others introduced the concept of Chevalley groups
over commutative rings.
The foundations of the theory of Chevalley groups have been laid  in the papers of Chevalley, Tits, Borel, Weil, Grothendieck, Demazure,  Stenberg, etc. In 1956--1958 Chevalley obtained a classification of semisimple algebraic groups over algebraically closed fields. Later on, Chevalley showed that all semisimple groups over an algebraically closed field are actually defined under~$\mathbb Z$, or, in other words, are obtained as a result of expanding to an arbitrary ring of some group scheme defined over~$\mathbb Z$. These group schemes are called \emph{Chevalley-Demazure schemes}. The groups of points of Chevalley–Demazure schemes over commutative rings are called \emph{Chevalley groups}. Chevalley groups include classical linear groups (special linear $\SL$, special orthogonal $\SO$, symplectic $\Sp$,
spinor $\Spin$, and also projective groups connected with them) over commutative rings. Finite simple groups of Lie type are  the central quotients of Chevalley groups.

Isomorphisms and automorphisms of Chevalley groups over different classes of rings were were intensively studied.
The description of isomorphisms of Chevalley groups over fields was obtained by
Steinberg~\cite{Stb1} for the finite case and by Humphreys~\cite{H} for the infinite one. Many papers are devoted 
to description of automorphisms of Chevalley groups over 
commutative rings. We can mention here the papers of
Borel--Tits~\cite{v22}, Carter--Chen~Yu~\cite{v24},
Chen~Yu~\cite{v25}--\cite{v29}, Abe~\cite{Abe_OSN}, Klyachko~\cite{Klyachko}.

Usually complete description of automorphisms of Chevalley groups means standardity of all these automorphisms, that is, all automorphisms are compositions of some simple and well-described types of automorphisms: inner automorphisms, automorphisms induced by ring automorphisms, etc.

Abe in~\cite{Abe_OSN} proved the standardity of automorphisms for Noetherian rings with~$1/2$, which could help to close the question of automorphisms of Chevalley groups over arbitrary commutative rings with~$1/2$. However, in considering the case of adjoint elementary groups has a gap, which cannot be eliminated by the methods of this article. 

The cases when the ring contains a lot of invertible integers (in some sense) are completely clarified in the paper of  Klyachko~\cite{Klyachko}. 

In the paper  \cite{ravnyekorni} Bunina proved that automorphisms of adjoint elementary Chevalley groups
with root systems $\mathbf A_l,\mathbf D_l, \mathbf E_l$, $l\geqslant 2$, over local rings with invertible $2$
can be represented as the composition of ring automorphism and an \emph{automorphism--conjugation} (by automorphism-conjugation we call  conjugation of elements of a Chevalley group in the adjoint representation by some matrix from
the normalizer of this group in  $\GL(V)$). By the similar token it was proved in \cite{normalizers} that every automorphism of an arbitrary  Chevalley (or its arbitrary subgroup) group  is standard, i.\,e.,
it is  a composition of ring, inner, central and graph automorphisms.
In the same paper it was obtained the theorem describing the normalizer of Chevalley groups in their adjoint representation,
which also holds for local rings without $1/2$.

In the series of papers \cite{bunF4}, \cite{korni2}, \cite{BunBl}, \cite{without2}, \cite{Bunina-Werewkin2} the similar methods made it possible to obtain the standardity of all automorphisms of Chevalley groups $G(\Phi,R)$ where $\Phi=\mathbf F_4$, $\mathbf B_l$, $l\geqslant 3$, $R$ is a local ring and $1/2\in R$, or $\Phi =\mathbf G_2$  and $1/2$, $1/3 \in R$. The same is true for $\Phi=\mathbf A_l$, $\mathbf D_l$ $\mathbf E_l$, $\mathbf G_2$, $l\geqslant 2$, $R$ is a local ring and $1/2\notin R$.
As we already mentioned the case $\mathbf C_l$ (symplectic linear groups and projective symplectic linear groups) was considered in the papers of Petechuk and Golubchik--Mikhalev (even for non-commutative rings). 

The non-standard automorphisms are described by Steinberg in~\cite{Steinberg} for the cases of Chevalley groups of types $\mathbf B_2$ and $\mathbf F_4$ over fields of characteristic~$2$ and of type $\mathbf G_2$  over fields of characteristic~$3$. For fields of characteristic~$2$ also there exists an isomorphism between Chevalley groups of types $\mathbf B_l$ and $\mathbf C_l$, $l\geqslant 3$. In \cite{Petechuk2} Petechuk described (non-standard) automorphisms of Chevalley groups of the type $\mathbf A_2$ over local rings without~$1/2$. Therefore the cases of Chevalley groups of the types $\mathbf A_2, \mathbf B_l, \mathbf C_l, \mathbf F_4$ over rings without~$1/2$ and of the type~$\mathbf G_2$ over rings without~$1/3$ require separate consideration.

In the paper \cite{Bunina_main} Bunina used the localization method and ideas of Petechuk and generalized the description of automorphisms of Chevalley groups over local rings to adjoint Chevalley groups over arbitrary commutative rings. In the paper~\cite{Bunina_recent} the isomorphisms between these Chevalley groups were described. 

In this paper we extend the result of~ \cite{Bunina_main} to arbitrary Chevalley groups over rings.

The paper is organized as follows. Section 2 deals with definitions and formulation of the Main Theorem. The proof of the Main Theorem for elementary case is situated in Section 3. The next Section 4 is devoted to the proof of the Main Theorem in the general case. 

\newpage

\section{Definitions and main theorem.}\leavevmode

\subsection{Root systems and semisimple Lie algebras}\leavevmode

We fix an indecomposable root system~$\Phi$ of the rank $\ell > 1$, with the system of simple
roots~$\Delta$, the set of positive (negative) roots $\Phi^+$
($\Phi^-$), and the Weil group~$W$. Recall that  any two roots of the same length are conjugate under the action of the Weil group. Let $|\Phi^+|=m$. More detailed texts about root systems and their
properties can be found in the books \cite{Hamfris}, \cite{Burbaki}.

Recall also that for $\alpha,\beta\in \Phi$
$$
\langle \alpha,\beta\rangle =2\frac{(\alpha,\beta)}{(\beta,\beta)}.
$$

Suppose now that we have a semisimple complex Lie algebra~$\mathcal
L$ with the Cartan subalgebra~$\mathcal H$ (more details about
semisimple Lie algebras can be found, for instance, in the book~\cite{Hamfris}).

Lie algebra  $\mathcal L$ has a decomposition ${\mathcal
L}={\mathcal H} \oplus \sum\limits_{\alpha\ne 0} {\mathcal
L}_\alpha$,
$$
{\mathcal L}_\alpha:=\{ x\in {\mathcal L}\mid [h,x]=\alpha(h)x\text{
for every } h\in {\mathcal H}\},
$$
and if ${\mathcal L}_\alpha\ne 0$, then $\dim {\mathcal
L}_\alpha=1$, all nonzero $\alpha\in {\mathcal H}$ such that
${\mathcal L}_\alpha\ne 0$, form some root system~$\Phi$. The root
system $\Phi$ and the semisimple Lie algebra
 $\mathcal L$ over~$\mathbb C$
uniquely (up to automorphism) define each other.

On the Lie algebra $\mathcal L$ we can introduce a bilinear
\emph{Killing form} $\varkappa(x,y)=\tr (\ad x\ad y),$  that is
non-degenerated on~$\mathcal H$. Therefore we can identify the
spaces $\mathcal H$ and ${\mathcal H}^*$.

We can choose a basis $\{ h_1, \dots, h_l\}$ in~$\mathcal H$ and for
every $\alpha\in \Phi$ elements $x_\alpha \in {\mathcal L}_\alpha$
so that $\{ h_i; x_\alpha\}$ is a basis in~$\mathcal L$ and for
every two elements of this basis their commutator is an integral
linear combination of the elements of the same basis. This basis is
called a \emph{Chevalley basis}.

\subsection{Elementary Chevalley groups}\leavevmode

Introduce now elementary Chevalley groups (see~\cite{Steinberg}).

Let  $\mathcal L$ be a semisimple Lie algebra (over~$\mathbb C$)
with a root system~$\Phi$, $\pi: {\mathcal L}\to \mathfrak{gl}(V)$
be its finitely dimensional faithful representation  (of
dimension~$n$). If $\mathcal H$ is a Cartan subalgebra of~$\mathcal
L$, then a functional
 $\lambda \in {\mathcal H}^*$ is called a
 \emph{weight} of  a given representation, if there exists a nonzero vector $v\in V$
 (that is called a  \emph{weight vector}) such that
for any $h\in {\mathcal H}$ $\pi(h) v=\lambda (h)v.$

In the space~$V$ in the Chevalley basis all operators
$\pi(x_\alpha)^k/k!$ for $k\in \mathbb N$ are written as integral
(nilpotent) matrices. An integral matrix also can be considered as a
matrix over an arbitrary commutative ring with~$1$. Let $R$ be such
a ring. Consider matrices $n\times n$ over~$R$, matrices
$\pi(x_\alpha)^k/k!$ for
 $\alpha\in \Phi$, $k\in \mathbb N$ are included in $M_n(R)$.

Now consider automorphisms of the free module $R^n$ of the form
$$
\exp (tx_\alpha)=x_\alpha(t)=1+t\pi(x_\alpha)+t^2
\pi(x_\alpha)^2/2+\dots+ t^k \pi(x_\alpha)^k/k!+\dots
$$
Since all matrices $\pi(x_\alpha)$ are nilpotent, we have that this
series is finite. Automorphisms $x_\alpha(t)$ are called
\emph{elementary root elements}. The subgroup in $\Aut(R^n)$,
generated by all $x_\alpha(t)$, $\alpha\in \Phi$, $t\in R$, is
called an \emph{elementary Chevalley group} (notation:
$E_\pi(\Phi,R)$).

In elementary Chevalley group we can introduce the following
important elements and subgroups:

\begin{itemize}
\item $w_\alpha(t)=x_\alpha(t) x_{-\alpha}(-t^{-1})x_\alpha(t)$, $\alpha\in \Phi$,
$t\in R^*$;

\item $h_\alpha (t) = w_\alpha(t) w_\alpha(1)^{-1}$;

\item  $N$ is generated by all
 $w_\alpha (t)$, $\alpha \in \Phi$, $t\in R^*$;

\item  $H$ is generated by all
 $h_\alpha(t)$, $\alpha \in \Phi$, $t\in R^*$;

\item The subgroup $U=U(\Phi,R)$ of the Chevalley group $G(\Phi,R)$ (resp. $E(\Phi,R)$) is generated by elements $x_\alpha(t)$, $\alpha\in \Phi^+$, $t\in R$, the subgroup $V=V(\Phi,R)$ is generated by elements $x_{-\alpha}(t)$, $\alpha\in \Phi^+$ $t\in R$.
\end{itemize}

The action of  $x_\alpha(t)$ on the Chevalley basis is described in
\cite{v23}, \cite{VavPlotk1}.

It is known that the group $N$ is a normalizer of~$H$ in elementary
Chevalley group, the quotient group $N/H$ is isomorphic to the Weil
group $W(\Phi)$.

All weights of a given representation (by addition) generate a
lattice (free Abelian group, where every  $\mathbb Z$-basis  is also
a $\mathbb C$-basis in~${\mathcal H}^*$), that is called the
\emph{weight lattice} $\Lambda_\pi$.

 Elementary Chevalley groups are defined not even by a representation of the Chevalley groups,
but just by its \emph{weight lattice}. More precisely, up to an abstract
isomorphism an elementary Chevalley group is completely defined by a
root system~$\Phi$, a commutative ring~$R$ with~$1$ and a weight
lattice~$\Lambda_\pi$.

Among all lattices we can mark two: the lattice corresponding to the
adjoint representation, it is generated by all roots (the \emph{root
lattice}~$\Lambda_{ad}$) and the lattice generated by all weights of
all reperesentations (the \emph{lattice of weights}~$\Lambda_{sc}$).
For every faithful reperesentation~$\pi$ we have the inclusion
$\Lambda_{ad}\subseteq \Lambda_\pi \subseteq \Lambda_{sc}.$
Respectively, we have the \emph{adjoint} and \emph{simply connected}
elementary Chevalley groups. 

Every elementary Chevalley group satisfies the following relations:

(R1) $\forall \alpha\in \Phi$ $\forall t,u\in R$\quad
$x_\alpha(t)x_\alpha(u)= x_\alpha(t+u)$;

(R2) $\forall \alpha,\beta\in \Phi$ $\forall t,u\in R$\quad
 $\alpha+\beta\ne 0\Rightarrow$
$$
[x_\alpha(t),x_\beta(u)]=x_\alpha(t)x_\beta(u)x_\alpha(-t)x_\beta(-u)=
\prod x_{i\alpha+j\beta} (c_{ij}t^iu^j),
$$
where $i,j$ are integers, product is taken by all roots
$i\alpha+j\beta$, taken in some fixed order; $c_{ij}$ are
integer numbers not depending on $t$ and~$u$, but depending on
$\alpha$ and $\beta$ and the order of roots in the product. 

(R3) $\forall \alpha \in \Phi$ $w_\alpha=w_\alpha(1)$;

(R4) $\forall \alpha,\beta \in \Phi$ $\forall t\in R^*$ $w_\alpha
h_\beta(t)w_\alpha^{-1}=h_{w_\alpha (\beta)}(t)$;

(R5) $\forall \alpha,\beta\in \Phi$ $\forall t\in R^*$ $w_\alpha
x_\beta(t)w_\alpha^{-1}=x_{w_\alpha(\beta)} (ct)$, where
$c=c(\alpha,\beta)= \pm 1$;

(R6) $\forall \alpha,\beta\in \Phi$ $\forall t\in R^*$ $\forall u\in
R$ $h_\alpha (t)x_\beta(u)h_\alpha(t)^{-1}=x_\beta(t^{\langle
\beta,\alpha \rangle} u)$.

For a given $\alpha\in \Phi$  by $X_\alpha$ we denote the subgroup   $\{ x_\alpha
(t)\mid t\in R\}$.

\subsection{Chevalley groups}\leavevmode

Introduce now Chevalley groups (see~\cite{Steinberg},
\cite{Chevalley}, \cite{v3}, \cite{v23}, \cite{v30}, \cite{v43},
\cite{VavPlotk1}, and references therein).

Consider semisimple linear algebraic groups over algebraically
closed fields. These are precisely elementary Chevalley groups
$E_\pi(\Phi,K)$ (see.~\cite{Steinberg},~\S\,5).

All these groups are defined in $\SL_n(K)$ as  common set of zeros of
polynomials of matrix entries $a_{ij}$ with integer coefficients
 (for example,
in the case of the root system $\mathbf C_\ell$ and the universal
representation we have $n=2l$ and the polynomials from the condition
$(a_{ij})Q(a_{ji})-Q=0$, where $Q$ is a matrix of the symplectic form). It is clear now that multiplication and
taking inverse element are  defined by polynomials with integer
coefficients. Therefore, these polynomials can be considered as
polynomials over an arbitrary commutative ring with a unit. Let some
elementary Chevalley group $E$ over~$\mathbb C$ be defined in
$\SL_n(\mathbb C)$ by polynomials $p_1(a_{ij}),\dots, p_m(a_{ij})$.
For a commutative ring~$R$ with a unit let us consider the group
$$
G(R)=\{ (a_{ij})\in \SL_n(R)\mid \widetilde p_1(a_{ij})=0,\dots
,\widetilde p_m(a_{ij})=0\},
$$
where  $\widetilde p_1(\dots),\dots \widetilde p_m(\dots)$ are
polynomials having the same coefficients as
$p_1(\dots),\dots,p_m(\dots)$, but considered over~$R$.

This group is called the \emph{Chevalley group} $G_\pi(\Phi,R)$ of
the type~$\Phi$ over the ring~$R$, and for every algebraically
closed field~$K$ it coincides with the elementary Chevalley group. In more advanced terms a Chevalley group $G(\Phi,R)$ is the value of the \emph{Chevalley-Demazure group scheme}, see \cite{??}. 

The subgroup of diagonal (in the standard basis of weight vectors)
matrices of the Chevalley group $G_\pi(\Phi,R)$ is called the
 \emph{standard maximal torus}
of $G_\pi(\Phi,R)$ and it is denoted by $T_\pi(\Phi,R)$. This group
is isomorphic to $Hom(\Lambda_\pi, R^*)$.

Let us denote by $h(\chi)$ the elements of the torus $T_\pi
(\Phi,R)$, corresponding to the homomorphism $\chi\in Hom
(\Lambda(\pi),R^*)$.

In particular, $h_\alpha(u)=h(\chi_{\alpha,u})$ ($u\in R^*$, $\alpha
\in \Phi$), where
$$
\chi_{\alpha,u}: \lambda\mapsto u^{\langle
\lambda,\alpha\rangle}\quad (\lambda\in \Lambda_\pi).
$$

\subsection{Connection between Chevalley groups and their elementary subgroups}\leavevmode

Connection between Chevalley groups and corresponding elementary
subgroups is an important problem in the structure theory of Chevalley
groups over rings. For elementary Chevalley groups there exists a
convenient system of generators $x_\alpha (\xi)$, $\alpha\in \Phi$,
$\xi\in R$, and all relations between these generators are well-known.
For general Chevalley groups it is not always true.

If $R$ is an algebraically closed field, then
$$
G_\pi (\Phi,R)=E_\pi (\Phi,R)
$$
for any representation~$\pi$. This equality is not true even for the
case of fields, which are not algebraically closed.

However if $G$ is a simply connected Chevalley group and the ring $R$ is \emph{semilocal}
(i.e., contains only finite number of maximal ideals), then we have
the condition
$$
G_{sc}(\Phi,R)=E_{sc}(\Phi,R).
$$
\cite{M}, \cite{Abe1}, \cite{St3}, \cite{v19}.

If, however, $\pi$ is arbitrary and $R$ is semilocal, then: $G_\pi
(\Phi,R)=E_\pi(\Phi,R)T_\pi(\Phi,R)$] (see~\cite{Abe1}, \cite{v19},
\cite{M}), and the elements $h(\chi)$ are connected with elementary
generators by the formula
\begin{equation}\label{e4}
h(\chi)x_\beta (\xi)h(\chi)^{-1}=x_\beta (\chi(\beta)\xi).
\end{equation}

\begin{remark}\label{Remark_torus}
Since $\chi\in \Hom (\Lambda(\pi),R^*)$, if we know the values of~$\chi$ on some set of roots which generate all roots (for example, on some basis of~$\Phi$), then we know $\chi(\beta)$ for all $\beta\in \Phi$ and respectively all $x_\beta(\xi)^{h(\chi)}$ for all $\beta\in \Phi$ and $\xi\in R^*$.

Therefore (in particular) if for all roots $\beta$ from some generating set of~$\Phi$ we have $[x_\beta(1),h(\chi)]=1$, then $h(\chi)\in Z(E_\pi(\Phi, R))$ and hence $h(\chi)\in Z(G_\pi(\Phi,R))$.

We will use this observation in the next section many times.
\end{remark}

If $\Phi$ is an irreducible root system of a rank $\ell\geqslant
2$, then $E(\Phi,R)$ is always normal and even {\bf characteristic} in $G(\Phi,R)$ (see~\cite{v41}, \cite{v35}). In the case of
semilocal rings it is easy to show that
$$
[G(\Phi,R),G(\Phi,R)]=E(\Phi,R).
$$
except the cases $\Phi=\mathbf B_2, \mathbf G_2$, $R=\mathbb F_2$.

In the case $\ell=1$ the subgroup of elementary matrices
$E_2(R)=E_{sc}(\mathbf A_1,R)$ is not necessarily normal in the special linear
group $\SL_2(R)=G_{sc}(\mathbf A_1,R)$ (see~\cite{Cn}, \cite{Sw},
\cite{v13}).

In the general case the difference between $G_\pi(\Phi,R)$ and $E_\pi (\Phi,R)$ is measured by $K_1$-functor.

\subsection{Standard automorphisms of Chevalley groups}\leavevmode

Define four types of automorphisms of a Chevalley group
 $G_\pi(\Phi,R)$, we
call them  \emph{standard}.

{\bf Central automorphisms.} Let $C_G(R)$ be a center of
$G_\pi(\Phi,R)$, $\tau: G_\pi(\Phi,R) \to C_G(R)$ be some
homomorphism of groups. Then the mapping $x\mapsto \tau(x)x$ from
$G_\pi(\Phi,R)$ onto itself is an automorphism of $G_\pi(\Phi,R)$,
 denoted by~$\tau$. It is called a \emph{central automorphism}
of the group~$G_\pi(\Phi,R)$.

{\bf Ring automorphisms.} Let $\rho: R\to R$ be an automorphism of
the ring~$R$. The mapping $(a_{i,j})\mapsto (\rho (a_{i,j}))$ from $G_\pi(\Phi,R)$
onto itself is an automorphism of the group $G_\pi(\Phi,R)$, 
denoted by the same letter~$\rho$. It is called a \emph{ring
automorphism} of the group~$G_\pi(\Phi,R)$. Note that for all
$\alpha\in \Phi$ and $t\in R$ an element $x_\alpha(t)$ is mapped to
$x_\alpha(\rho(t))$.

{\bf Inner automorphisms.} Let $S$ be some ring containing~$R$,  $g$
be an element of $G_\pi(\Phi,S)$, that normalizes the subgroup $G_\pi(\Phi,R)$. Then
the mapping $x\mapsto gxg^{-1}$  is an automorphism
of the group~$G_\pi(\Phi,R)$, denoted by $i_g$.  It is called an
\emph{inner automorphism}, \emph{induced by the element}~$g\in G_\pi(\Phi,S)$. If $g\in G_\pi(\Phi,R)$, then we call $i_g$ a \emph{strictly inner}
automorphism.

{\bf Graph automorphisms.} Let $\delta$ be an automorphism of the
root system~$\Phi$ such that $\delta \Delta=\Delta$. Then there
exists a unique automorphisms of $G_\pi (\Phi,R)$ (we denote it by
the same letter~$\delta$) such that for every $\alpha \in \Phi$ and
$t\in R$ an element $x_\alpha (t)$ is mapped to
$x_{\delta(\alpha)}(\varepsilon(\alpha)t)$, where
$\varepsilon(\alpha)=\pm 1$ for all $\alpha \in \Phi$ and
$\varepsilon(\alpha)=1$ for all $\alpha\in \Delta$.

Now suppose that $\delta_1,\dots, \delta_k$ are all different graph automorphisms for the given root system  (for the systems $\mathbf E_7,\mathbf E_8,\mathbf B_l,\mathbf C_l,\mathbf F_4,\mathbf G_2$ there can be just identical automorphism, for the systems $\mathbf A_l,\mathbf  D_l,  l\ne 4, \mathbf E_6$ there are two such
automorphisms, for the system $\mathbf D_4$ there are six automorphisms). Suppose that we have a system of orthogonal idempotents of
the ring~$R$:
$$
\{\varepsilon_1, \dots, \varepsilon_k\mid \varepsilon_1+\dots+\varepsilon_k=1, \forall i\ne j\ \varepsilon_i\varepsilon_j=0\}.
$$
Then the mapping
$$
\Lambda_{\varepsilon_1,\dots,\varepsilon_k}:= \varepsilon_1 \delta_1+\dots+ \varepsilon_k \delta_k
$$
of the Chevalley group onto itself is an  automorphism,  called a \emph{graph automorphism} of the Chevalley
group $G_\pi (\Phi,R)$.

 Similarly we can define four types of automorphisms of the elementary
subgroup~$E_\pi(\Phi,R)$. An automorphism~$\sigma$ of the group
 $G_\pi(\Phi,R)$ (or $E_\pi(\Phi,R)$)
is called  \emph{standard} if it is a composition of automorphisms
of these introduced four types.

In \cite{Bunina_main} the following theorem was proved:

\begin{theorem}\label{main_adjoint}
Let $G=G_{\ad}(\Phi,R)$ 
be an adjoint Chevalley group (or its elementary subgroup $(E_{\ad}(\Phi,R))$) of rank $>1$,  $R$ be a commutative ring with~$1$. Suppose that for $\Phi = \mathbf A_2, \mathbf B_l, \mathbf C_l$ or $\mathbf F_4$ we have $1/2\in R$, for $\Phi=\mathbf G_2$ we have $1/2,1/3 \in R$. Then every automorphism of the group~$G$
is standard and the inner automorphism in the composition is strictly inner.
\end{theorem}

Our goal is to prove the following theorem:

\begin{theorem}\label{main_sc}
Let $G=G_{\pi}(\Phi,R)$ 
be a  Chevalley group \emph{(}or its elementary subgroup $E_{\pi}(\Phi,R)))$ of rank $>1$,  $R$ be a commutative ring with~$1$. Suppose that for $\Phi = \mathbf A_2, \mathbf B_l, \mathbf C_l$ or $\mathbf F_4$ we have $1/2\in R$, for $\Phi=\mathbf G_2$ we have $1/2,1/3 \in R$. Then every automorphism of the group~$G$
is standard.
\end{theorem}

\section{Proof of the main theorem for elementary Chevalley groups and subgroups}\leavevmode

\subsection{Localization of rings and modules; injection of a ring into the product of its localizations.}\leavevmode

\begin{definition}  Let $R$ be a commutative ring. A subset $Y\subset R$ is called \emph{multiplicatively closed} in~$R$, if $1\in Y$ and $Y$ is closed under multiplication.
\end{definition}

Introduce  an equivalence relation $\sim$ on the set of pairs $R\times Y$ as follows:
$$
\frac{a}{s}\sim \frac{b}{t} \Longleftrightarrow \exists u\in Y:\ (at-bs)u=0.
$$
  By $\frac{a}{s}$ we denote the whole equivalence class of the pair $(a,s)$, by $Y^{-1}R$ we denote the set of all equivalence classes. On the set $S^{-1}R$ we can introduce the ring structure by
$$
\frac{a}{s}+\frac{b}{t}=\frac{at+bs}{st},\quad \frac{a}{s}\cdot \frac{b}{t}=\frac{ab}{st}.
$$

\begin{definition}
The ring $Y^{-1}R$ is called the \emph{ring of fractions of~$R$ with respect to~$Y$}.
\end{definition}

 Let $\mathfrak p$ be a prime ideal of~$R$. Then the set $Y=R\setminus {\mathfrak p}$ is multiplicatively closed (it is equivalent to the definition of the prime ideal). We will denote the ring of fractions  $Y^{-1}R$ in this case by $R_{\mathfrak p}$. The elements $\frac{a}{s}$, $a\in \mathfrak p$, form an ideal $\mathfrak M$ in~$R_{\mathfrak p}$. If $\frac{b}{t}\notin \mathfrak M$, then $b\in Y$, therefore $\frac{b}{t}$ is invertible in~$R_{\mathfrak p}$. Consequently the ideal $\mathfrak M$ consists of all non-invertible elements of the ring~$R_{\mathfrak p}$, i.\,e., $\mathfrak M$ is the greatest ideal of this ring, so $R_{\mathfrak p}$ is a local ring.

The process of passing from~$R$ to~$R_{\mathfrak p}$ is called  \emph{localization at~${\mathfrak p}$.}

\begin{proposition}\label{inlocal}
Every commutative ring  $R$ with $1$ can be naturally embedded in the cartesian product of all its localizations  by maximal ideals
 $$
S=\prod\limits_{{\mathfrak m}\text{ is a maximal ideal of }R} R_{\mathfrak m}
$$
by diagonal mapping, which corresponds every $a\in R$ to the element
$
\prod\limits_{\mathfrak m} \left( \frac{a}{1}\right)_{\mathfrak m} \in S.
$
\end{proposition}

\subsection{Proof for $E_{\pi}(\Phi,R)$.}\leavevmode

Suppose that $G=G_{\pi}(\Phi,R)$ or $G=E_{\pi}(\Phi,R)$ is a Chevalley group (or its elementary subgroup), where $\Phi$ is an indecomposable root system of rank $>1$, $R$ is an arbitrary commutative ring (with $1/2$ in the case $\Phi =\mathbf A_2, \mathbf F_4, \mathbf B_l, \mathbf C_l$ and with $1/2$ and $1/3$ in the case $\Phi=\mathbf G_2$). Suppose that $\varphi\in \Aut (G)$.

Since the subgroup $E_{\pi}(\Phi,R)$ is characteristic in $G_{\pi}(\Phi,R)$, then  $\varphi$ induces the automorphism $\varphi \in \Aut (E_{\pi}(\Phi,R))$ (we denote it by the same letter). 

The elementary adjoint Chevalley group $E_{\ad}(\Phi, R)$ is the quotient group of our initial elementary Chevalley group $E_{\pi}(\Phi, R)$ by its center~$Z=Z(E_\pi(\Phi,R))$. Therefore the automorphism $\varphi$ induces an automorphism $\overline \varphi$ of the adjoint Chevalley group $E_{\ad}(\Phi, R)$. By Theorem~\ref{main_adjoint} $\varphi$ is the composition of a graph automorphism $\overline \Lambda_{\varepsilon_1,\dots, \varepsilon_k}$, where $\varepsilon_1,\dots, \varepsilon_k\in R$, a ring automorphism $\overline \rho$, induced by $\rho\in \Aut R$, and the strictly inner automorphism $i_{\overline g}$, induced by some $\overline g\in G_{\ad}(\Phi,R)$. Central automorphism is identical in the decomposition of~$\overline \varphi$, since the center of any adjoint Chevalley group is trivial.

Since $\varepsilon_1,\dots, \varepsilon_k\in R$ and for any $\delta_i\in \Aut \Delta$ and for any representation~$\pi$ of the corresponding Lie algebra there exists the corresponding graph automorphism $\delta_i\in \Aut (G_\pi (\Phi, R))$, then there exists a graph automorphism $\Lambda_{\varepsilon_1,\dots, \varepsilon_k}\in \Aut (E_{\pi}(\Phi,R))$ such that the induced automorphism of the group $E_{\ad}(\Phi,R)$ is precisely~$\overline \Lambda_{\varepsilon_1,\dots, \varepsilon_k}$. 

Also taking the ring automorphism $\rho\in \Aut (G_{\pi}(\Phi,R))$ we see that the induced automorphism of $E_{\ad}(\Phi,R)$ is precisely~$\overline \rho$.

Therefore if we take $\varphi_1=  \Lambda^{-1}\circ \rho^{-1}\circ \varphi$, then we obtain an automorphism of the group $G$ (and in any cases of the group/subgroup $E_{\pi}(\Phi,R)$) which induces the strictly inner automorphism $i_{\overline g}$ on $E_{\ad}(\Phi,R)$. 

We always assume that $R$ is a subring of the ring $S=\prod\limits_{{\mathfrak m}} R_{\mathfrak m}=\prod\limits_{i\in \varkappa} R_i$, where every $R_i$ is a local ring, therefore 
$$
G_{\pi}(\Phi, R)\subseteq G_{\pi} (\Phi,S)=\prod_{i\in \varkappa} G_{\pi}(\Phi, R_i)\text{ and }E_{\pi}(\Phi, R)\subseteq \prod_{i\in \varkappa} E_{\pi}(\Phi, R_i).
$$
Note that since every $R_i$ is local, then  we have $G_\pi (\Phi,R)=T_\pi(\Phi,R) E_\pi (\Phi,R)$ and therefore
$$
\prod\limits_{i\in \varkappa} G_{\pi}(\Phi,R_i)=\prod\limits_{i\in \varkappa} T_\pi(\Phi,R_i)E_{\pi}(\Phi,R_i).
$$

Suppose now that $\overline g = \prod\limits_{i\in \varkappa} \overline g_i$, where $\overline g_i\in G_{\ad}(\Phi,R_i)$.

Let us consider one $i\in \varkappa$, where $\overline g_i \in T_{\ad}(\Phi, R_i) E_{\ad}(\Phi,R_i)$, i.\,e., $g_i=\overline t_i\cdot \overline x_i$, where $\overline t_i \in T_{\ad}(\Phi,R_i)$, $\overline x_i \in E_{\ad} (\Phi,R_i)$.

Since $\overline x_i$ is a product of elementary unipotents over the ring~$R_i$, then we can take $x_i\in E_\pi(\Phi,R_i)$, that is the same product of the same elementary unipotents and its image under factorization of $E_\pi (\Phi,R_i)$ by its center is precisely~$\overline x_i$.

Now let us consider the element $\overline t_i\in T_{\ad}(\Phi,R_i)$. This element  corresponds to some homomorphism $\chi_i \in \Hom (\Lambda (\ad), R_i^*)$ and acts on any $x_\alpha(s)\in E_{\ad}(\Phi, R_i)$ as
$$
\overline t_i x_\alpha(s) \overline t_i^{-1}=x_\alpha(\chi_i(\alpha)\cdot s).
$$
If $\overline t_i \notin H_{\ad}(\Phi,R_i)$, then we can extend the ring $R_i$ up to a ring $S_i$ so that there exists $\overline h_i\in H_{\ad}(\Phi,S_i)$ with the same action on all elementary uniponents $x_\alpha (s)$ as our~$\overline t_i$. The ring $S_i$ is an algebraic extension of~$R_i$, in which there exist several new roots $\sqrt[k]{\lambda}$ for a finite number of $\lambda \in R_i^*$. This $S_i$ can be obtained from~$R_i$ by the standard procedure 
$$
S_i\cong R_i[y]/(y^k-\lambda).
$$
Note that $S_i$ is not necessarily local.

Now since $R_i\subseteq S_i$, then $S\subseteq \prod\limits_{i\in \varkappa} S_i=\widetilde S$ and $R\subseteq S\subseteq \widetilde S$. We see that for every $i\in \varkappa$ the torus element
$\overline t_i$ acts on all $x_\alpha(s)$, $s\in S_i$ as $\overline h_i \in H_{\ad} (\Phi, S_i)$, therefore the element $\overline y_i=\overline h_i\cdot \overline x_i$ acts on all $x_\alpha(s)$, $s\in S_i$ as the initial~$\overline g_i$. 

Consequently the element $\overline y:= \prod\limits_{i\in \varkappa} \overline y_i \in E_{\ad}(\Phi, \widetilde S)$ acts on all $x_\alpha(s)$, $s\in \widetilde S$ as the initial~$\overline g$. 

Therefore we have $\overline y \in E_{\ad}(\Phi,\widetilde S)$ such that 
$$
i_{\overline y}\vert_{E_{\ad}(\Phi,\widetilde S)}=i_{\overline g}\vert_{E_{\ad}(\Phi,\widetilde S)}.
$$
In particular,
$$
i_{\overline y}\vert_{E_{\ad}(\Phi,R)}=i_{\overline g}\vert_{E_{\ad}(\Phi,R)}.
$$
Let us take $y\in E_{\pi}(\Phi,\widetilde S)$ such that its image under factorization of $E_{\pi}(\Phi,\widetilde S)$ by its center is precisely~$\overline y$. 

Now we can take $\varphi_2=i_{ y^{-1}}\circ \varphi_1$, it will be an isomorphism between $E_\pi (\Phi, R)$ and the subgroup of $E_\pi (\Phi,\widetilde S)$ such that under factorization by the center of $E_\pi (\Phi, \widetilde S)$ we obtain the identical automorphism $\overline \varphi_2$ of the group $E_{\ad} (\Phi,R)$.

Now let us analyze the mapping $\varphi_2$. 

Since $\overline \varphi_2$ is identical, then 
$$
\forall \alpha\in \Phi \ \forall s\in R\quad \varphi_2(x_\alpha(s))=z_{\alpha,s} x_\alpha(s),\text{ where }z_{\alpha,s}\in Z(E_\pi(\Phi,\widetilde S)).
$$
If $\alpha$ is either any root of the systems $\mathbf A_l$, $l\geqslant 2$, $\mathbf D_l$, $l\geqslant 4$, $\mathbf E_l$, $l=6,7,8$, $\mathbf F_4$, or any long root of the systems $\mathbf G_2$, $\mathbf B_l$, $l\geqslant 3$,  or any short root of the systems $\mathbf C_l$, $l\geqslant 3$, then $\alpha$ can be represented as $\alpha=\beta+\gamma$, where $\{ \pm \beta, \pm \gamma, \pm \alpha\} \cong \mathbf A_2$. In this case 
$$
x_\alpha (s)=[x_\beta(s),x_\gamma(1)],
$$
therefore
$$
z_{\alpha,s} x_\alpha(s)=\varphi_2(x_\alpha(s))=[\varphi_2(x_\beta(s)),\varphi_2(x_\gamma(1))]=[z_{\beta,s}x_\beta(s),z_{\gamma,1}x_\gamma(1)]=[x_\beta(s),x_\gamma(1)]=x_\alpha(s).
$$
Consequently, $z_{\alpha,s}=1$ for all $s\in R$.

For the root system $\mathbf G_2$ all Chevalley groups are adjoint and so we do not need to prove Theorem~1 for this root system.

For the root system $\mathbf B_2$ if $\alpha$ is a long simple root and $\beta$ is a short simple root, then $\Phi^+=\{ \alpha, \beta,\alpha+\beta,\alpha+2\beta\}$, where $\alpha+\beta$ is short and $\alpha+2\beta$ is long and
\begin{align*}
[x_\alpha(t), x_\beta(u)]&=x_{\alpha+\beta}(\pm tu)x_{\alpha+2\beta}(\pm tu^2),\\
[x_{\alpha+\beta}(t),x_\beta(u)]&=x_{\alpha+2\beta}(\pm 2tu).
\end{align*}
(see \cite{Steinberg}, Lemma~33). 

Since for the root system $\mathbf B_2$ we require $1/2\in R$, then 
$$
[x_{\alpha+\beta}(s),x_\beta(1/2)]=x_{\alpha+2\beta}(\pm s)
$$
and by the same arguments as above $z_{\gamma,s}=1$ for all long roots $\gamma$ and all $s\in R$.
Then
\begin{multline*}
x_{\alpha +\beta}(\pm s)x_{\alpha+2\beta}(\pm s)=[x_\alpha (s), x_\beta(1)]=\varphi_2([x_\alpha (s), x_\beta(1)])=\\
=\varphi_2(x_{\alpha +\beta}(\pm s)x_{\alpha+2\beta}(\pm s))=z_{\alpha+\beta,\pm s}x_{\alpha +\beta}(\pm s)x_{\alpha+2\beta}(\pm s),
\end{multline*}
thus $z_{\gamma,s}=1$ also for all short roots $\gamma\in \mathbf B_2$. Therefore for $\mathbf B_2$ for all $\alpha\in \Phi$ and all $s\in R$ the mapping $\varphi_2$ is an identical automorphism of $E_\pi (\Phi,R)$.

Since any root $\gamma$ of the root system $\mathbf B_l$ or $\mathbf C_l$, $l\geqslant 3$, can be embedded to some root system isomorphic to~$\mathbf B_2$, and in this case we also require $1/2\in R$, then for these root systems also $z_{\gamma,s}=1$ for all $s\in R^*$ and $\varphi_2$ is an identical automorphism of $E_\pi (\Phi,R)$.

Therefore for all cases under consideration
$$
\varphi_2\vert_{E_\pi(\Phi, R)}=i_{y^{-1}}\circ \Lambda^{-1} \circ \rho^{-1} \circ \varphi \vert_{E_\pi(\Phi, R)}=id_{E_\pi(\phi,R)},
$$
so
$$
\varphi \vert_{E_\pi(\Phi, R)}=\rho \circ \Lambda \circ i_y \vert_{E_\pi(\Phi, R)},
$$
where $y\in E_\pi (\Phi,\widetilde S)\cap N (E_\pi (\Phi,R))$, $\Lambda$ is a graph automorphism of the groups $G_\pi(\Phi,R)$ and $E_\pi (\Phi,R)$ and $\rho$ is a ring automorphism of the groups $G_\pi(\Phi,R)$ and $E_\pi (\Phi,R)$.

Thus, for $G=E_\pi (\Phi,R)$ the main theorem (Theorem~2) is proved.

\section{Proof of the main theorem for the groups $G_\pi(\varphi,R)$}

Let now $G=G_\pi (\Phi,R)$. Initially the mapping $\varphi$ was an automorphism of the group~$G$. The mapping $\varphi_1$ from the previous section was the composition of $\varphi$ and graph and ring automorphisms of the group~$G$, i.\,e., also an automorphism of~$G$. After that $\varphi_2$ (from the previous section)  is the composition of $\varphi_1$ and the conjugation of~$G$ by some element $y\in E_\pi(\Phi,\widetilde S)$, where $R\subset \widetilde S$. We know that $y$ normalizes~$E_\pi(\Phi,R)$ and we want to show that in our case $y$ normalizes also our full Chevalley group~$G$.

Note that for the simply-connected Chevalley group of the type $\mathbf E_6$ Luzgarev and Vavilov in~\cite{Normalizer_E6} proved that the normalizers of the Chevalley group and its elementary subgroup coincide. Then in~\cite{Normalizer_E7} they proved the same theorem for the root system~$\mathbf E_7$. Since all other exceptional Chevalley groups are adjoint, we only need to show the coincidence of normalizers for non-adjoint classical Chevalley groups, but our method will cover all the cases.

\begin{lemma}\label{generate1}
Under assumptions of Theorem~\ref{main_sc} the elements $x_\alpha(1)$, $\alpha\in \Phi$, by addition, multiplication and multiplication by elements from~$R$
generate the Lie algebra $\pi({\mathcal L}_R(\Phi))\subset M_N(R)$,  where $N$ is the dimension of the representation~$\pi$.
\end{lemma}
\begin{proof}
For the adjoint Chevalley groups this lemma was proved in~\cite{Bunina_main}.
Therefore we will not repeat the proof for the root system~$\mathbf G_2$ (since it is always adjoint).

If the root system differs from $\mathbf G_2$ and $1/2\in R$, then $x_\alpha (1)=E+\pi(X_\alpha)  + \pi(X_\alpha)^2/2$, therefore
$$
\pi(X_\alpha)=x_\alpha(1)-E-(x_\alpha(1)-E)^2/2,
$$
and 
$$
\pi({\mathcal L}_R(\Phi))=\langle \pi(X_\alpha)\mid \alpha\in \Phi\rangle_R.
$$

Suppose now that we deal with systems $\mathbf A_l$, ($l\geqslant 3$),  $\mathbf D_l, \mathbf E_l$, $1/2\notin R$.

For all these systems and non-adjoint representations~$\pi$ we have $\pi (X_\alpha)^2=0$ for all $\alpha\in \Phi$, therefore 
$$
\pi(X_\alpha)=x_\alpha(1)-E.
$$

The lemma is proved.

\end{proof}

  From Lemma \ref{generate1} we see that the conjugation by~$y$ maps the Lie algebra $\pi(\mathcal L(\Phi)_R)$ onto itself.

\begin{lemma}\label{generate2}
Under assumptions of Theorem~\ref{main_sc} the Lie algebra $\pi({\mathcal L}_R(\Phi))$ together with the unity matrix~$E$  by addition, multiplication and multiplication by elements from~$R$
generate the matrix ring $M_N(R)$, where $N$ is the dimension of the representation~$\pi$.
\end{lemma}

\begin{proof}
For all adjoint Lie algebras under consideration this fact was proved in the papers~\cite{ravnyekorni}, \cite{korni2}, \cite{BunBl}, \cite{bunF4},~\cite{without2}.

For classical representations of classical Lie algebras the proof is clear and direct:

{\bf 1.} If we have the root system $\mathbf A_l$ and the standard representation, then
$$
\pi(X_{e_i-e_j})=E_{ij},\quad \pi(X_{e_i-e_j})\pi(X_{e_j-e_i})=E_{ii},\quad M_{l+1}(R)=\langle E_{ij}\mid 1\leqslant i,j\leqslant l+1\rangle_R.
$$

{\bf 2.} The Lie algebra of the type $\mathbf C_l$ in its universal representation has $2l$-dimensional linear space and the basis
$$
\{ E_{ii}-E_{l+i,l+i}; E_{ij}-E_{l+j,l+i}; E_{i,l+i}; E_{l+i,i}; E_{i,l+j}+E_{j,l+i}; E_{l+i,j}+E_{l+j,i}\mid 1\leqslant i\ne j\leqslant l\}.
$$
Multiplying $E_{ij}-E_{l+j,l+i}$ by $E_{j,l+j}$, we get all $E_{i, l+j}$ for all $1\leqslant i,j \leqslant l$. Multiplying $E_{l+i,i}$ by $E_{ij}-E_{l+j,l+i}$, we obtain $E_{l+i,j}$ for all $1\leqslant i,j \leqslant l$. It is clear that after that we have all $E_{ij}$, $1\leqslant i,j \leqslant l$, and therefore the whole matrix ring $M_{2l}(R)$.

{\bf 3.} For the root system $\mathbf D_l$ 
 the standard representation gives the algebra $\mathfrak{so}_{2l}$, where in $2l$-dimensional space the basis is
 $$
\{ E_{ii}-E_{l+i,l+i}; E_{ij}-E_{l+j,l+i}; E_{i,l+j}-E_{j,l+i}; E_{i+l,j}-E_{j+l,i}\mid 1\leqslant i\ne j\leqslant l\}.
$$
Since for $i\ne j$ we have $(E_{ii}-E_{l+i,l+i})\cdot (E_{ij}-E_{l+j,l+i})=E_{ij}$, then the whole matrix ring $M_{2l}(R)$ is generated by this Lie algebra.

\medskip
All other representations are described by Plotkin, Semenov and Vavilov in~\cite{Atlas} as \emph{microweight} representations with the help of so-called \emph{weight diagrams}. 

Weight diagram is a labeled graph, its vertices correspond (bijectively) to the weights $\lambda \in \Lambda (\pi)$. The vertices corresponding to $\lambda, \mu\in \Lambda (\pi)$, are joined by a bond marked $\alpha_i\in \Delta$ (or simply~$i$) if and only if their difference $\lambda - \mu=\alpha_i$ is a simple root. The diagrams are usually drawn in such way that the marks on the opposite (parallel) sides of a parallelogram are equal and at least one of them is usually omitted. All weights are numbered in any order and give the basis of our representation~$\pi$. If we want to find $\pi (X_{\alpha_i})$, $i=1,\dots, l$, then we need to find all bonds marked by~$i$, and if they join the vertices $(\gamma_1, \gamma_1+\alpha_i),\dots, (\gamma_k,\gamma_k+\alpha_i)$, then 
$$
\pi (X_{\alpha_i})=\pm E_{\gamma_1,\gamma_1+\alpha_i} \pm \dots \pm E_{\gamma_k,\gamma_k+\alpha_i},\quad \pi (X_{-\alpha_i})=\pm E_{\gamma_1+\alpha_i,\gamma_1} \pm \dots \pm E_{\gamma_k+\alpha_i,\gamma_k}.
$$
It is clear that if we take an element $\pi (X_{\alpha_i})\cdot \pi (X_{\alpha_j})$, then it is a sum of $\pm E_{\gamma,\gamma'}$, where there exists a path from the weight $\gamma$ to~$\gamma'$ of the length~$2$ marked by the sequence $(i,j)$. Similarly,  if we take an element $\pi (X_{\alpha_{i_1}})\times   \dots \times \pi (X_{\alpha_{i_k}})$, then it is a sum of $\pm E_{\gamma,\gamma'}$, where there exists a path from the weight $\gamma$ to~$\gamma'$ of the length~$k$ marked by the sequence $(i_1,\dots, i_k)$.

Our goal is to generate all matrix units $E_{\gamma_1,\gamma_2}$, where $\gamma_1,\gamma_2\in \Lambda(\pi)$. Since all weight diagrams are connected, it is sufficient to generate all matrix units $E_{\gamma,\gamma+\alpha_i}$ and $E_{\gamma+\alpha_i,\gamma}$, where $\alpha_i\in \Delta$, $\gamma, \gamma+\alpha_i\in \Lambda(\pi)$.
The general idea how to do it is the following: for any $\gamma\in \Lambda(\pi)$ and any $\alpha_{i_0}\in \Delta$ such that $\gamma+\alpha_{i_0}\in \Lambda(\pi)$ we find $\gamma'\in \Lambda(\pi)$ such that:

(1) there exists a path $(i_0,i_1,\dots, i_k)$ from $\gamma$ to $\gamma'$;

(2) in our weight diagram there is no other path $(i_0,i_1,\dots, i_k)$;

(3) the path $(i_1,\dots, i_k)$ exists only from $\gamma+\alpha_i$ to~$\gamma'$.

Then 
$$
\pi(X_{\alpha_{i_0}})\pi (X_{\alpha_{i_1}})\dots \pi (X_{\alpha_{i_k}})=\pm E_{\gamma,\gamma'}
$$ 
and
$$
 \pi (X_{-\alpha_{i_k}})\dots \pi (X_{-\alpha_{i_1}})=\pm E_{\gamma',\gamma+\alpha_{i_0}}
$$
and therefore $E_{\gamma,\gamma+\alpha_{i_0}}=E_{\gamma,\gamma'}E_{\gamma',\gamma+\alpha_{i_0}}$.

It is almost clear that such $\gamma'$ and  unique paths always exist, we will just show one diagram as an example.

\begin{figure}
\caption{${\bf A}_7$, $\omega_2$}

\begin{picture}(200,120)
\put(110,20){\line(1,0){30}}
\put(140,20){\line(1,1){75}}
\put(140,20){\line(1,-1){15}}
\put(155,5){\line(1,1){75}}
\put(155,35){\line(1,-1){30}}
\put(185,5){\line(1,1){60}}
\put(170,50){\line(1,-1){45}}
\put(215,5){\line(1,1){45}}
\put(185,65){\line(1,-1){60}}
\put(200,80){\line(1,-1){75}}
\put(215,95){\line(1,-1){75}}
\put(245,5){\line(1,1){30}}
\put(275,5){\line(1,1){15}}
\put(290,20){\line(1,0){30}}
\put(110,20){\circle*{4}}
\put(140,20){\circle*{4}}
\put(170,20){\circle*{4}}
\put(200,20){\circle*{4}}
\put(230,20){\circle*{4}}
\put(260,20){\circle*{4}}
\put(290,20){\circle*{4}}
\put(320,20){\circle*{4}}
\put(155,5){\circle*{4}}
\put(185,5){\circle*{4}}
\put(215,5){\circle*{4}}
\put(245,5){\circle*{4}}
\put(275,5){\circle*{4}}
\put(155,35){\circle*{4}}
\put(185,35){\circle*{4}}
\put(215,35){\circle*{4}}
\put(245,35){\circle*{4}}
\put(275,35){\circle*{4}}
\put(170,50){\circle*{4}}
\put(200,50){\circle*{4}}
\put(230,50){\circle*{4}}
\put(260,50){\circle*{4}}
\put(185,65){\circle*{4}}
\put(215,65){\circle*{4}}
\put(245,65){\circle*{4}}
\put(200,80){\circle*{4}}
\put(230,80){\circle*{4}}
\put(215,95){\circle*{4}}
\put(120,10){$\mathbf 2$}
\put(140,3){$\mathbf 1$}
\put(170,3){$\mathbf 2$}
\put(200,3){$\mathbf 3$}
\put(222,3){$\mathbf 5$}
\put(252,3){$\mathbf 6$}
\put(282,3){$\mathbf 7$}
\put(303,10){$\mathbf 6$}
\put(140,30){$\mathbf 3$}
\put(155,45){$\mathbf 4$}
\put(170,60){$\mathbf 5$}
\put(185,75){$\mathbf 6$}
\put(200,90){$\mathbf 7$}
\put(225,90){$\mathbf 1$}
\put(240,75){$\mathbf 2$}
\put(255,60){$\mathbf 3$}
\put(270,45){$\mathbf 4$}
\put(285,30){$\mathbf 5$}
\put(100,10){$\gamma_1$}
\put(130,10){$\gamma_2$}
\put(150,-5){$\gamma_3$}
\put(150,25){$\gamma_4$}
\put(165,25){$\gamma_5$}
\put(165,40){$\gamma_6$}
\put(180,-5){$\gamma_7$}
\put(180,25){$\gamma_8$}
\put(180,55){$\gamma_9$}
\put(195,25){$\gamma_{10}$}
\put(195,40){$\gamma_{11}$}
\put(195,70){$\gamma_{12}$}
\put(210,-5){$\gamma_{13}$}
\put(210,25){$\gamma_{14}$}
\put(210,55){$\gamma_{15}$}
\put(210,85){$\gamma_{16}$}
\put(225,25){$\gamma_{17}$}
\put(225,40){$\gamma_{18}$}
\put(225,70){$\gamma_{19}$}
\put(240,-5){$\gamma_{20}$}
\put(240,25){$\gamma_{21}$}
\put(240,55){$\gamma_{22}$}
\put(255,25){$\gamma_{23}$}
\put(255,40){$\gamma_{24}$}
\put(270,-5){$\gamma_{25}$}
\put(270,25){$\gamma_{26}$}
\put(285,10){$\gamma_{27}$}
\put(315,10){$\gamma_{28}$}
\end{picture}
\end{figure} 

If we take the case $\mathbf A_7$ with the weight $\omega_2$, the representation is $28$-dimensional. Let us find a path which gives $E_{\gamma_1,\gamma_2}$. Since the path $(1,3)$ is unique in the diagram, then the path $(2,1,3)$ is also unique and we have 
$$
E_{\gamma_1,\gamma_2}=(\pi (X_{\alpha_2})\pi(X_{\alpha_1})\pi(X_{\alpha_3}))\cdot (\pi(X_{-\alpha_3})\pi(X_{-\alpha_1}).
$$
If we want to generate, for example, $E_{\gamma_4,\gamma_6}$, then the suitable path is $(4,1,5)$, since the path $(1,5)$ is unique in the diagram.

Looking at the picture it is easy to find the suitable path for any pair of neighboring vertices.

Therefore the lemma is proved for all the cases.
\end{proof}

Since $y \pi(\mathcal L (\Phi)_R)y^{-1}=\pi(\mathcal L (\Phi)_R$ and $\pi(\mathcal L (\Phi)_R$ generates the whole matrix ring $M_N(R)$, then $yM_N(R)y^{-1}=M_N(R)$. Therefore $y G_\pi(\Phi,R)y^{-1}\subseteq \SL_N(R)$. From the other side, since $y\in G_\pi (\Phi,\widetilde S)$, then $y G_\pi(\Phi,R)y^{-1}\subseteq G_\pi (\Phi,\widetilde S)$. Since $G_\pi (\Phi,\widetilde S) \cap \SL_N(R)$ is (by definition) the Chevalley group $G_\pi (\Phi,R)$, then $y$ normalizes~$G$.

Now we know that $\varphi_2$ is an automorphism of~$G=G_\pi(\Phi,R)$, identical on the elementary subgroup $E=E_\pi(\Phi,R)$. Let us take some $g\in G$ and $x_1\in E$ and let $gx_1g^{-1}=x_2\in E$. Then
$$
\varphi_2(g) \varphi_2(x_1) \varphi_2(g)^{-1}=\varphi_2(x_2)\Longrightarrow \varphi_2(g) x_1 \varphi_2(g)^{-1}=x_2,
$$
therefore
$$
  \varphi_2(g) x_1 \varphi_2(g)^{-1} =g x_1 g^{-1}\Longrightarrow (g^{-1}\varphi_2(g)) x_1 (g^{-1}\varphi_2(g))^{-1}=x_1,
$$
so
$$
g^{-1}\varphi_2(g) \in C_G (E).
$$
By the main theorem from~\cite{v18} $C_G(E)=Z(G)$, therefore
$$
\varphi_2(g)=c_g \cdot g,\quad c_g\in Z(G)\text{ for all } g\in G.
$$

Whence $\varphi_2$ is a central automorphism of~$G$ and the initial $\varphi$ is the composition of graph, ring, inner and central automorphisms, i.\,.e, $\varphi$ is standard.

The theorem is proved.

\section{Some applications: isomorphisms and model theory of Chevalley groups}\leavevmode

Standard description of automorphisms of Chevalley groups allows to describe and classify Chevalley groups up to different type of equivalencies and also to study model-theoretic properties. 


\begin{theorem}\label{isom_general}
Let $G_1=G_{\pi_1}(\Phi_1,R_1)$ and $G_2=G_{\pi_2}(\Phi_2,R_2)$ 
be two  Chevalley groups of ranks $>1$,  $R_1$, $R_2$ be  commutative rings with~$1$. Suppose that for $\Phi_1 = \mathbf A_2, \mathbf B_l, \mathbf C_l$ or $\mathbf F_4$ we have $1/2\in R_1$, for $\Phi_1=\mathbf G_2$ we have $1/2,1/3 \in R_1$. Then every isomorphism between the groups~$G_1$ and~$G_2$ is standard: it is a composition of  inner, diagram and central automorphisms of~$G_1$ and ring isomorphism between $G_1$ and~$G_2$.
\end{theorem}

\begin{proof}
The proof  is identical to the proof of Theorem~9 from~\cite{Bunina_recent}. One needs to replace the references to Theorem~1 in the proof by those to Theorem~2. 
\end{proof}

\begin{remark}
 The result of Theorem~\ref{isom_general} is valid with respect to elementary Chevalley groups $E_{\pi_1}(\Phi_1,R_1)$ and $E_{\pi_2}(\Phi_2,R_2)$ as well. 
\end{remark}

\begin{corollary}[classification of Chevalley groups up to isomorphism]
 Under conditions from Theorem~\ref{isom_general} two Chevalley groups  $G_1$ and $G_2$ (elementary Chevalley groups, respectively) are isomorphic if and only if they have the same root systems $\Phi_1$ and~$\Phi_2$, same weight lattices $\Lambda_{\pi_1}$ and $\Lambda_{\pi_2}$ and isomorphic rings $R_1$ and~$R_2$.
\end{corollary}

\begin{proof}
If $G_1\cong G_2$, then there exists an isomorphism $\varphi : G_1\to G_2$, which is composition of a ring isomorphism $\rho: G_1\to G_2$ and some automorphism $\psi\in \Aut G_1$ (according to Theorem~\ref{isom_general}). Therefore there exists a ring isomorphism between $G_1$ and~$G_2$, i.\,e., $G_1$ and $G_2$ have the same root systems, weight lattices and isomorphic rings.
\end{proof}

Another application of Theorem~\ref{isom_general} is classification of Chevalley groups up to elementary equivalence (for adjoint Chevalley groups it was done in~\cite{Bunina_recent}).

\begin{definition}
Two algebraic systems $\mathcal M_1$ and $\mathcal M_2$ of the same language~$\mathcal    L$ are called \emph{elementarily equivalent}, if their first order theories coincide.
\end{definition}

\begin{theorem}[Keisler--Shelah Isomorphism theorem, \cite{chapter1-21}, \cite{chapter1-20}]\label{isomorp}
Two models ${\mathcal M}_1$ and ${\mathcal M}_2$ of the same language are elementarily equivalent if and only if there exists an ultrafilter~${\mathcal F}$ such that
$$
\prod_{\mathcal F} {\mathcal M}_1\cong \prod_{\mathcal F}{\mathcal M}_2.
$$
\end{theorem}

\begin{corollary}[classification of Chevalley groups up to elementary equivalence]
Under conditions from Theorem~\ref{isom_general} two Chevalley groups  $G_1$ and $G_2$ (elementary Chevalley groups, respectively) are elementarily equivalent if and only if they have the same root systems $\Phi_1$ and~$\Phi_2$, same weight lattices $\Lambda_{\pi_1}$ and $\Lambda_{\pi_2}$ and elementarily equivalent rings $R_1$ and~$R_2$.
\end{corollary}

\begin{proof}
By Theorem~\ref{isomorp} the groups $G_1$ and $G_2$ are elementarily equivalent if and only if for some ultrafilter~$\mathcal F$ their ultrapowers are isomorphic. Since
$$
\prod_{\mathcal F} G_\pi(\Phi,R) \cong G_\pi (\Phi, \prod_{\mathcal F} R),
$$
the latter is equivalent to
$$
G_{\pi_1} (\Phi_1, \prod_{\mathcal F}R_1)\cong G_{\pi_2} (\Phi_2, \prod_{\mathcal F}R_2) \Longleftrightarrow 
\begin{cases}
\Lambda_{\pi_1} = \Lambda_{\pi_2},\\
\Phi_1=\Phi_2,\\
\prod_{\mathcal F} R_1\cong \prod_{\mathcal F} R_2,
\end{cases}
\Longleftrightarrow 
\begin{cases}
\Lambda_{\pi_1} = \Lambda_{\pi_2},\\
\Phi_1=\Phi_2,\\
 R_1\equiv R_2,
\end{cases}
$$
what was required.
\end{proof}

Two last corollaries almost finalize classification of Chevalley groups over commutative rings up to isomorphisms and elementary equivalence. However,  there are still open questions concerning the relations  of Chevalley groups with model theory. 

In the recent work of D.\,Segal and K.\,Tent~\cite{Segal-Tent} the question of bi-interpretability of Chevalley groups over integral domains was considered (see \cite{Segal-Tent} and \cite{Khael-Miasn} for the definition of \emph{bi-interpretability}):

\begin{theorem}[\cite{Segal-Tent}]\label{Segal-Tent}
Let $G(R)=G_\pi(\Phi,R)$ be a Chevalley group of rank at
least two, and let $R$ be an integral domain. Then $R$ and $G(R)$ are bi-interpretable
provided either

\emph{(1)} $G$ is adjoint, or

\emph{(2)} $G(R)$ has finite elementary width,\\
assuming in case $\Phi=\mathbf E_6, \mathbf E_7, \mathbf E_8$, or $\mathbf F_4$ that $R$ has at least two units.
\end{theorem}

In the paper \cite{Bunina-bi}  \emph{regular} bi-interpretabilty of Chevalley groups over local rings was obtained. This result used the ideas from~\cite{Segal-Tent} along with description of isomorphisms between Chevalley groups over local rings. It has also been proved that the class of Chevalley groups over local rings is \emph{elementarily definable}: \emph{any group that is elementarily equivalent to some Chevalley group over a local ring is also a Chevalley group (of the same type) over a local ring} (see~ \cite{Bunina-bi}). Theorem~2 and~\ref{isom_general} of the current paper allows us to prove regular bi-interptretability and elementary definability of adjoint Chevalley groups and Chevalley groups  of finite elementary width over arbitrary commutative rings.

\medskip

{\bf Acknowledgements.}
Our sincere thanks go to Eugene Plotkin for very useful discussions regarding various aspects of this work and permanent attention to it.

\end{document}